\documentclass{agtart_a}
\pdfoutput=1

\usepackage[all]{xy}
\usepackage[T1]{fontenc}

\title{Subrings of singular cohomology associated to spectra}
\author{Geoffrey M\,L Powell}
\givenname{Geoffrey M\,L}
\surname{Powell}
\address{Laboratoire Analyse\\
G\'eom\'etrie et Applications, UMR 7539\\ 
Institut Galil\'ee\\\newline
Universit\'e Paris 13\\
93430 Villetaneuse\\France}
\email{powell@math.univ-paris13.fr} 
\urladdr{http://www.math.univ-paris13.fr/~powell}

\volumenumber{6}
\issuenumber{}
\publicationyear{2006}
\papernumber{38}
\startpage{1037}
\endpage{1067}

\doi{}
\MR{}
\Zbl{}


\received{7 July 2005}
\revised{}
\accepted{24 May 2006}
\published{24 August 2006}
\publishedonline{24 August 2006}
\proposed{}
\seconded{}
\corresponding{}
\editor{}
\version{}

\arxivreference{}  

\makeatletter
\def\cnewtheorem#1[#2]#3{\newtheorem{#1}{#3}[subsection]
\expandafter\let\csname c@#1\endcsname\c@thm}
\makeatother

\let\xysavmatrix\xymatrix
\def\xymatrix{\disablesubscriptcorrection\xysavmatrix}
\AtBeginDocument{\let\bar\wbar\let\tilde\wtilde\let\hat\what}
\AtBeginDocument{\renewcommand{\hom}{\mathrm{Hom}}}

\newenvironment{ta}{\par\medskip{\bf Acknowledgement}\qua }{}

\newtheorem*{THM-A}{Theorem A}
\newtheorem*{THM-B}{Theorem B}

\swapnumbers
\newtheorem{thm}{Theorem}[subsection]
\cnewtheorem{prop}[thm]{Proposition}
\cnewtheorem{cor}[thm]{Corollary}
\cnewtheorem{lem}[thm]{Lemma}

\cnewtheorem{quest}[thm]{Question}
\cnewtheorem{conj}[thm]{Conjecture}
\cnewtheorem{proper}[thm]{Properties}
\cnewtheorem{conv}[thm]{Convention}

\theoremstyle{definition}
\newtheorem{defn}[thm]{Definition}
\newtheorem{exam}[thm]{Example}

\theoremstyle{remark}
\newtheorem{rem}[thm]{Remark}
\newtheorem{nota}[thm]{Notation}
\newtheorem{hyp}[thm]{Hypothesis}

\newcommand{\llabel}[1]{\label{#1}}
\newcommand{\reset}{\setcounter{thm}{0}}
\newcommand{\topspace}{{\mathcal{T}op}}
\newcommand{\hurewicz}[1][E]{{\mathbb{H}_{#1}}}
\newcommand{\hurewiczeven}[1][E]{{\mathbb{H}_{#1}^{\mathrm{ev}}}}
\newcommand{\esp}{\mathbf{E}}
\newcommand{\presh}[1][\calc]{{\mskip2mu\widehat{\mskip-2mu#1\mskip-2mu}\mskip2mu}}
\newcommand{\pfpresh}[1][\calc]{{{\mskip2mu\widehat{\mskip-2mu#1\mskip-2mu}\mskip2mu^{\mathrm{pf}}}}}
\newcommand{\set}{\mathrm{Set}}
\newcommand{\pfset}{{\mathcal{PS}}}
\newcommand{\mapc}{\mathrm{Map}_c}
\newcommand{\chern}[1][E]{\mathrm{Ch}_{#1}}
\newcommand{\vs}{\mathcal{E}}
\newcommand{\fdvs}{{\mathcal{E}^{\mathrm{f}}}}
\newcommand{\field}{\mathbb{F}}
\newcommand{\restrict}[1][n]{{\langle #1 \rangle}}
\newcommand{\f}{\mathcal{F}}
\newcommand{\g}{\mathcal{G}}
\newcommand{\op}{^\mathrm{op}}
\newcommand{\unstmod}{\mathcal{U}}
\newcommand{\unstalg}[1][]{{\mathcal{K}_{#1}}}
\newcommand{\nil}{{\mathcal{N}il}}
\newcommand{\rep}{\mathrm{Rep}}
\newcommand{\acat}[1][n]{\mathcal{A}_p^{(#1)}}

\newcommand{\cala}{\mathcal{A}}
\newcommand{\calc}{\mathcal{C}}
\newcommand{\cald}{\mathcal{D}}

\newcommand{\zed}{\mathbb{Z}}

\begin{document}

\begin{asciiabstract}
This paper extends the relation established for group cohomology by Green,
Hunton and Schuster between chromatic phenomena in stable homotopy theory
and certain natural subrings of singular cohomology. This exploits the
theory due to Henn, Lannes and Schwartz of unstable algebras over the
Steenrod algebra localized away from nilpotents.
\end{asciiabstract}

\begin{abstract}
This paper extends the relation established for group cohomology by Green,
Hunton and Schuster between chromatic phenomena in stable homotopy theory
and certain natural subrings of singular cohomology. This exploits the
theory due to Henn, Lannes and Schwartz of unstable algebras over the
Steenrod algebra localized away from nilpotents.
\end{abstract}

\maketitle

\section{Introduction}

There is a relation between the chromatic filtration of stable
homotopy theory and a stratification  of the algebraic variety associated to
the cohomology of a finite group which was made explicit by
the work of Green, Hunton and Schuster \cite{ghs}.  The
purpose of this paper is to investigate generalizations of this
theory, by exploiting the theory of unstable algebras over the
Steenrod algebra for singular cohomology with coefficients in a prime field. 

The motivation for the constructions was provided by the Chern ring for group
cohomology. A unitary representation $G
\rightarrow U(n)$  of a finite group $G$ induces a morphism of classifying spaces $BG
\rightarrow BU$ by passage to the infinite unitary group. The  induced morphism in singular
cohomology $H^* (BU) \rightarrow H^*
(BG)$ is a morphism of rings and the Chern subring
$\chern[]^* (BG) \subset H^* (BG)$  is defined to be  the
subring generated by the image. 

The space $BU$ is the connected component of
the base point of the space $\Omega^\infty KU$ associated to  unitary
$K$--theory. This suggests a natural definition of a Chern subring
associated to an arbitrary  spectrum $E$; a further generalization is
obtained by considering any topological space in place of $BG$. For a general
topological space $X$, there are  inclusions of algebras
\[
\chern ^*  (X) 
\hookrightarrow 
H^* _E (X) 
\hookrightarrow 
H^* (X),
\]
where $\chern ^* (X)$ generalizes the Chern subring and $H^*_E (X) $ is
 a larger subring defined by a generalization of the construction.

When singular cohomology is taken with coefficients in a prime field
$\field$,  these are morphisms of unstable algebras over the Steenrod algebra and
are therefore amenable to study from this point of view. Green, Hunton
and Schuster  studied the ring $\chern ^* (BG)$; for certain generalized
cohomology theories,  the associated variety is determined 
 in terms of a category related to the Quillen category of
elementary abelian $p$--groups associated to the  finite group $G$ \cite{ghs}.  The  theory  of unstable algebras over the Steenrod algebra localized
away from nilpotents of Henn, Lannes and Schwartz \cite{hls}, which is equivalent to working up to
$F$--isomorphism, allows a partial generalization of the theory from
classifying spaces to more general spaces.

The functor $X\mapsto H^*_E (X)$ depends upon the
topological space $X$ and not only upon the cohomology ring $H^* (X)$
as an object of the category $\unstalg$ of unstable algebras over the Steenrod algebra. For $E$ a complex-oriented cohomology theory, there is an algebraic approximation to $H_E^* (X)$; namely, for  an integer $n$ related to the height of the mod--$p$ reduction of the formal group associated to the spectrum $E$,  there is a  functor $\alpha \restrict\co \unstalg \rightarrow \unstalg$ such that  the functor $H^*_E$ is related to $\alpha
\restrict H^*$. 

The general result is given in  \fullref{cor-canon-F-mono}; it is stated below for the $I_n$--complete Johnson-Wilson spectrum, $\widehat{E(n)}$. Recall that a morphism of unstable algebras is an $F$--monomorphism if every element of the kernel is nilpotent.

\begin{THM-A}
For $n$ a positive integer and  $X$  a topological space for which the Hure\-wicz homomorphism
$[B\field ^n,X] \rightarrow \hom_{\unstalg}(H^* (X), H^*
(B\field ^n))$ is a surjection,   there is a natural morphism of unstable algebras over the Steenrod algebra
$$ H^* _{\tiny\widehat{E(n)}} (X) 
\rightarrow
\alpha \restrict H^* (X) 
$$
which is an $F$--monomorphism.
\end{THM-A}

The universal example of the Eilenberg--MacLane
spaces shows that this morphism is not  an $F$--isomorphism in general.
 The result of Green, Hunton and Schuster \cite{ghs} shows that it is an
 $F$--isomorphism for $X= BG$, the classifying space of a finite
 group; this uses the work  of Hopkins, Kuhn
and Ravenel \cite{hkr} on generalized characters to give the  construction of enough elements in the
generalized cohomology of a finite group. 

The strategy generalizes to the consideration of the Borel cohomology
of  a $G$--space  $X$ to yield the following result.

\begin{THM-B}
For $G$  a finite group, $X$  a finite $G$--CW--complex  and 
$n$  a nonnegative integer,  there is an $F$--isomorphism of
algebras
$$
H^*_{\tiny\widehat{E(n)}} (EG\times _G X) 
\cong_F
\alpha \restrict H^* (EG \times _G X).
$$
\end{THM-B}

\begin{ta}
This paper was motivated by discussions with John Hunton
 during his visits to Universit\'e Paris 13 as  `Professeur
 Invité'  in  2002.
\end{ta}


\section{Unstable modules and unstable algebras}
\label{Sec-unstable}

This section provides a rapid overview of the theory of localization
of unstable algebras over the Steenrod algebra away from
nilpotents. This theory gives the appropriate model for the algebraic
variety associated to an unstable algebra. This overview introduces the
notation which is used throughout the paper, whereas the reader is
referred to the sources for details of the theory.

Let $\cala$ denote the Steenrod algebra over $\field$, the prime field
of characteristic $p$, which is fixed throughout this paper.  The category of unstable modules over 
$\cala$ is denoted by  $\unstmod$  and  the subcategory of unstable
algebras by $\unstalg$   (see Schwartz \cite{s} and Lannes \cite{l}). Singular cohomology
is always taken with $\field$--coefficients; it takes values in the
category $\unstalg$ of unstable algebras.

\subsection{Localization away from nilpotents}
The theory of Henn, Lannes and Schwartz of localization away from
nilpotent unstable modules is a fundamental tool in studying
cohomology rings up to $F$--isomorphism (see \cite{hls,k,s}).

The full subcategory of nilpotent unstable modules is written $\nil$;
localization of the abelian category $\unstmod$ away from the
nilpotent modules yields the abelian quotient category $\unstmod / \nil$. There is the corresponding construction for the category $\unstalg/ \nil$, in which the $F$--isomorphisms become invertible \cite[page 1077]{hls}. 

An unstable module  $M \in \unstmod$ is naturally the colimit of its
finitely generated submodules in $\unstmod$, hence the vector space
$\hom_\unstmod (M, H^* (BV))$ has a natural profinite structure for
$V$ an elementary abelian $p$--group; $\hom_\unstmod (M,H^* (BV))'$
denotes the continuous dual.  Similarly, an unstable algebra $K \in \unstalg$ is naturally the filtered colimit of its
subobjects in $\unstalg$ which are finitely generated as
algebras over the Steenrod algebra, hence  the set $\hom_{\unstalg} (K, H^* (BV))$ has
a natural profinite structure.

\begin{nota}$\phantom{9}$
\begin{enumerate}
\item
Let $\vs$ denote the category of $\field$--vector spaces and let $\fdvs$
denote the full subcategory of finite-dimensional vector
spaces.
\item
Let $\f$ denote the category of functors from $\fdvs$ to $\vs$ and let 
 $\f_\omega \subset \f$ denote the full subcategory of analytic functors. (A
functor is analytic if it is the colimit of its finite subobjects,
where an object of $\f$  is termed finite if it has a composition series of
finite length).
\item
Let $\pfset$ denote the category of profinite sets and  continuous
morphisms, which  are denoted by $\mapc ( -, -)$. 
\item
Let $\g$ denote the category of functors from $\fdvs$ to $\pfset\op$; thus 
 $\g \op$ is equivalent to the category of functors from
$\fdvs\op$ to $\pfset$.
\vspace{10pt}
\item
\begin{enumerate}
\item
Let $f\co \unstmod \rightarrow \f$ denote the functor
$\{V \mapsto \hom_{\unstmod}( -, H^* (BV)) '\}$; 
\item
 let $g\co \unstalg \rightarrow \g$ denote the functor $\{ V \mapsto
 \hom_{\unstalg} (-, H^* (BV))\}$.
\end{enumerate}
\end{enumerate}
\end{nota}

The functors $f,g$ fit into a commutative diagram in which $\Theta$ denotes the forgetful functor:
\[
\xymatrix{
\unstalg \ar[r]^g
\ar@{^(->}[d]_\Theta
&
\g
\ar[d]^{\hom_{\pfset\op}(\field, -)} 
\\
\unstmod
\ar[r]_f
&
\f.
}
\]
\begin{rem}
It is frequently convenient to consider the functor $g\op \co \unstalg \op \rightarrow \g \op$ in place of $g$; the functor $g\op$ corresponds to a contravariant functor on $\unstalg$ with values in the category of presheaves of profinite sets on $\fdvs$.
\end{rem}

\begin{defn}
 \cite{hls}\qua An object $G \in \g$ is
analytic if the functor $\hom_{\pfset\op}(\field, G)$ is analytic in $\f$.  The full subcategory of analytic functors in $\g$ is denoted by $\g_\omega$.
\end{defn}

The fundamental result of Henn, Lannes, and Schwartz \cite{hls} is summarized in the following theorem.

\begin{thm}
\label{thm:hls}
{\rm \cite{hls}}\qua
The functors $f, g$ induce equivalences of categories:
\begin{eqnarray*}
&& 
f \co \unstmod / \nil
 \stackrel{\cong}{\rightarrow}
 \f _\omega;\\
&&g \co\unstalg / \nil
\stackrel{\cong}{\rightarrow}
\g_\omega.
\end{eqnarray*}
\end{thm}

\begin{rem}
The functors which induce the inverse equivalences are induced by the
respective right adjoints $ a\co \g \rightarrow \unstalg $  and $m \co \f
\rightarrow \unstmod$ to the functors $g\co \unstalg
\rightarrow \g$ and $f\co\unstmod \rightarrow \f$ (see \cite{hls}).
\end{rem}



\section{Subalgebras of singular cohomology}
\label{Sec-heconst}
\reset

This section recalls the definition of the generalized Chern subring
which is associated to a spectrum, following Green, Hunton and
Schuster \cite{ghs}.  For later convenience, various quotients of the
set of homotopy classes $[Z, X]$ associated to cohomology functors are
introduced using the Hurewicz morphism. The relationships between
these functors are explored.

\begin{nota}
Throughout this section,  let $E$ denote a spectrum; the associated
$\Omega$--spectrum is denoted $\esp_\bullet$. 
\end{nota}

\subsection{Generalized Chern subrings}
Let $\topspace$ denote a  `convenient' category of topological
spaces; this can be taken to be modelled by the category of simplicial sets.

\begin{defn}\cite{ghs}\qua
For $X \in \topspace$ a topological space, let 
\begin{enumerate}
\item
$H^* _E (X)$ denote the subring of $H^* (X) $ generated by cohomology
classes represented by (unpointed) morphisms $X \rightarrow \esp_s \rightarrow
K (\field, t)$, for some integers $s,t$;
\item
$\chern^* (X)$ denote the subring of $H^* (X) $ generated by cohomology
classes represented by (unpointed) morphisms $X \rightarrow \esp_{2s} \rightarrow
K (\field, t)$, for some integers $s,t$.
\end{enumerate}
\end{defn}

\begin{rem}
The ring $\chern ^* (X)$ is referred to as the generalized Chern
subring; the definition is motivated by the case $X =BG$, the
classifying space of a finite group,  with  $E=KU$ (see \cite[Proposition 1.6 and Remark 1.8]{ghs}).
\end{rem}

\begin{prop}
For $X \in \topspace$, the algebras $H^*_E (X) $ and  $\chern ^* (X)$ have
unique unstable algebra structures over the Steenrod algebra for which the monomorphisms  
$$\chern^* (X) \hookrightarrow H^* _E (X)  \hookrightarrow H^* (X) $$
are morphisms in $\unstalg$. 

Moreover, the associations $X \mapsto H^* _E (X)$ and 
$X \mapsto \chern^* (X)$  are contravariantly functorial in $X$. 
\end{prop}

\begin{proof}
Straightforward.
\end{proof}

\subsection{The unstable Hurewicz homomorphism}

Singular cohomology induces a natural unstable Hurewicz
transformation
\[
h_E \co E^\bullet (X) 
\rightarrow 
\hom _{\unstalg} ( H^* (\esp_\bullet) , H^* (X)). 
\]
\begin{lem}
Via the natural inclusions $\chern^* (X) \hookrightarrow H_E^*  (X)
\hookrightarrow H^* (X)$, the transformation $h_E \co E^{\bullet} (X) 
\rightarrow \hom_{\unstalg} (H^* (\esp_{\bullet} ) , H^* (X))$ induces
natural transformations:
\begin{eqnarray*}
&&h_E \co E^{\bullet} (X) 
\rightarrow \hom_{\unstalg} (H^* (\esp_{\bullet} ) , H^*_E (X));
\\
&&h_E \co E^{2\bullet} (X) 
\rightarrow \hom_{\unstalg} (H^* (\esp_{2\bullet} ) , \chern ^*(X))
\end{eqnarray*}
\end{lem}

\begin{proof}
Straightforward.
\end{proof}

\subsection[Quotients associated to cohomology theories]{Quotients of $[Z, X]$ associated to cohomology theories}

\label{subsect:quotients-htpy}

Throughout this section, let $Z,X$ denote spaces in $\topspace$.
 \fullref{prop:QE} gives an explicit relation between the
 functors $E^*$ and $H^*_E$. There is an analogous result for the
 Chern subring $\chern^*$, by restriction to the even spaces of the
 $\Omega$--spectrum for $E$.

\begin{nota}$\phantom{9}$
\label{nota:E-HE}
\begin{enumerate}
\item
Let $[Z, X] _{E} $ denote the image of the morphism 
$$[Z, X]
\rightarrow \hom_{E^*} (E^* (X) , E^* (Z) ) $$
induced by $E^*$--cohomology  and let $q_E\co [Z, X]
\twoheadrightarrow [Z, X]_E$ denote the canonical surjection.
\item
Let $[Z, X] _{H_E} $ denote the image of the morphism 
$$[Z, X]
\rightarrow \hom_{\unstalg} (H^*_E (X) , H^*_E (Z) ) $$
 induced by the
functor $H^* _E$  and let $q_{H_E}\co [Z, X]
\twoheadrightarrow [Z, X]_{H_E}$ denote the canonical surjection.
\item 
Let $[Z, X] _{H} $ denote the image of the morphism 
$$[Z, X]
\rightarrow \hom_{\unstalg} (H^* (X) , H^* (Z) ) $$
 induced by the
functor $H^*$  and let $q_{H}\co [Z, X]
\twoheadrightarrow [Z, X]_{H}$ denote the canonical surjection.
\end{enumerate}
\end{nota}

\begin{rem}
The notation $[-,-]_H$ does not conflict with  $[-,-]_E$,
since the image of the morphism $[Z, X] \rightarrow \hom_{H^*} (H^*
(X), H^* (Z))$ necessarily takes values in $\hom_{\unstalg} (H^*
(X), H^* (Z))$.
\end{rem}

\begin{prop}
\label{prop:QE functorial}
The constructions $[-,-]_E$ and $[-,-]_{H_E}$ define functors 
$$
\topspace \op
\times \topspace \rightarrow \mathrm{Set}$$
 and the morphisms $q_E$ and
$q_{H_E}$ are natural transformations.
\end{prop}

\begin{proof}
Straightforward.
\end{proof}

\begin{lem}
\label{lem:Hurewicz-surj}
If the Hurewicz morphism $[Z, X] \rightarrow \hom_{\unstalg}
(H^* (X), H^* (Z)) $ is surjective, then there is a canonical
isomorphism
$$
[Z, X]_H \cong 
\hom_{\unstalg}
(H^* (X), H^* (Z))
.
$$
\end{lem}

\begin{proof}
Tautological.
\end{proof}

There is a restriction morphism 
$$\hom_{\unstalg} (H^* (X) , H^* (Z))
\rightarrow \hom_{\unstalg} (H^* _E (X), H^* (Z))$$ 
 induced by
the inclusion $H^* _E (X)\hookrightarrow H^* (X)$. However, this need
not factor across the inclusion $\hom_{\unstalg}  (H^* _E (X), H^*_E (Z))
\hookrightarrow \hom_{\unstalg}  (H^* _E (X), H^* (Z))$ which is induced by $H^*
_E (Z)\hookrightarrow H^* (Z)$. When restricting to the image of the
Hurewicz morphism, the situation is simpler.

\begin{prop}
\label{prop:HE}
There is a canonical surjection 
$
p_H\co [Z, X] _H \rightarrow [Z, X]_{H_E}
$
which factorizes the morphism $q_{H_E}\co [Z,X]\rightarrow [Z, X]_{H_E}$.
\end{prop}

\begin{proof}
For $f\co Z\rightarrow X$, the composite $H^* _E (X) \hookrightarrow H^*
(X) \stackrel{H^* (f)} {\longrightarrow}
H^* (Z)$ factors canonically across $H^* _E (f)\co H^* _E (X) \rightarrow H^*
_E(Z)$; this defines the map, which is surjective.
\end{proof}

\begin{prop}
\label{prop:QE}
There is a natural commutative diagram
$$
\xymatrix{
[Z, X] 
\ar@{->>}[r]^{q_E}
\ar@{->>}[d]_{q_{H_E}}
&
[Z, X]_E
\ar@{->>}[ld]^{Q_E}
\\
[Z, X]_{H_E}
}
$$
in which the map $Q_E$ is surjective.
\end{prop}

\begin{proof}
Straightforward.
\end{proof}



\section{Fundamental properties}
\label{Sec-fundamental}
The results of
Green, Hunton and Schuster \cite{ghs} depend on certain fundamental properties of the cohomology
theory $E^*$ which  are developed  axiomatically in this
section. These results are rephrased here by using the functors
$[-,-]_E$ and $[-,-]_{H_E}$ which were introduced in  \fullref{subsect:quotients-htpy}.

Throughout this section, $E$ denotes a fixed spectrum, $n$ denotes a
nonnegative integer, and $Z$ and $X$ denote topological spaces
in $\topspace$.

\subsection{The Hurewicz hypotheses}

The following hypothesis allows morphisms in $E$--cohomology to be detected in cohomology. 

\begin{hyp}
The Hurewicz hypothesis 
$\hurewicz (X)$ holds if  
$$h_E \co E^\bullet (X) \rightarrow \hom_{\unstalg} (H^*
(\esp_\bullet), H^* (X) )
$$
 is a monomorphism.
\end{hyp}

\begin{rem} There is an even space variant:  the even Hurewicz 
hypothesis $\hurewiczeven (X)$ holds if 
$h_E \co E^{2\bullet} (X) \rightarrow \hom_{\unstalg} (H^*
(\esp_{2\bullet}), H^* (X) )$ is a monomorphism.
\end{rem}

\begin{exam}
\label{exam:Kn}
\cite[Theorem 4.1]{ghs}\qua  Recall that the spectrum $\widehat{E(n)}$ is the complete version of the Johnson--Wilson theory $E(n)$ \cite{bw} and has coefficient ring  the completion of
 $\zed _{(p)} [v_ 1, \ldots , v_n , v_n ^{-1}]$ with respect
to the ideal $I_n \co= (p, v_1, \ldots , v_{n-1})$.  The
$I_n$--adically complete Johnson--Wilson spectrum $\widehat{E(n)}$ satisfies the hypothesis
$\hurewicz[\widehat{E (n)}] (BV)$, for any $V \in \fdvs$.
\end{exam}

The following Proposition is the fundamental role of the Hurewicz
hypothesis $\hurewicz (Z)$ in separating morphisms in
$E$--cohomology. Recall there is a natural transformation $Q_E \co [Z, X]_E
\rightarrow [Z, X]_{H_E}$ which was introduced in the statement of
\fullref{prop:QE}.

\begin{prop}
\label{prop:E-Hurewicz}
Suppose that  
 the Hurewicz hypothesis $\hurewicz (Z)$ holds,  then the map 
$$Q_E \co [Z, X]_E
\rightarrow [Z, X]_{H_E}
$$
is an isomorphism.
\end{prop}

\begin{proof}
It is sufficient to establish that the map is injective. Let $f_1, f_2
\co Z \rightrightarrows X$ be morphisms for which 
 the morphisms $
H^*_E (f_1) , H^*_E(f_2) \co 
H_E^* (X) 
\rightrightarrows
H_E^* (Z)
$
coincide and  consider the commutative diagram
\[
\xymatrix{
E^* (X) 
\ar@<1ex>[r]^{f_1^*}
\ar@<0ex>[r]_{f_2^*}
\ar[d]
 &
E^* (Z) 
\ar[d]
\\
\hom_{\unstalg} (H^* (\esp_\bullet) , H^*_E (X) ) 
\ar@<1ex>[r]
\ar@<0ex>[r]&
\hom_{\unstalg} (H^* (\esp_\bullet) , H^*_E (Z) ),
}
\]
where the bottom arrows are induced by $H^*_E (f_1) , H^*_E(f_2)$ and
hence coincide. The right hand vertical arrow is a monomorphism by hypothesis
$\hurewicz (Z)$; the result follows.
\end{proof}

\begin{rem}
\fullref{prop:E-Hurewicz} has an even space  variant in which
the hypothesis $\hurewicz (Z) $ is replaced by $\hurewiczeven (Z)$, the subring $H^* _E (-)$ is
replaced by the Chern subring $\chern ^* (-)$ and the statement is
restricted to $E^{2\bullet} (-)$.
\end{rem}

\subsection[Height--n detection]{Height--$n$ detection}
\label{subsect:height-n}

The examples of spectra $E$ which are of  interest in applications
correspond to complex-oriented
cohomology theories. For such theories, there is a description of the
cohomology $E^* (BV)$ of an elementary abelian $p$--group in terms of
the associated formal group. The work of Hopkins, Kuhn and
Ravenel \cite{hkr} highlights a fundamental property of a cohomology
theory, which is termed height--$n$ detection here. This is the key
ingredient to the constructions of this paper.

\begin{nota}
For  $\mathcal{H}$ a functor  from $\topspace \op$ to the category of
sets, $V$ a finite-dimensional $\field$--vector space and $n$ a
nonnegative integer, denote by
$$
\delta_{\mathcal{H}}\co 
\mathcal{H} (BV) 
\rightarrow 
\mathcal{H} (B \field ^n) ^{\hom (\field ^n , V)}
$$
the morphism adjoint to the evaluation morphism $\hom
(\field ^n , V) \times \mathcal{H} (BV) \rightarrow \mathcal{H}(B \field
  ^n)$. 
\end{nota}

\begin{defn}
The spectrum $E$  satisfies the height--$n$ detection
property if the morphism
\[
\delta_{E^*}\co
E^* (BV) 
\rightarrow 
E^* (B \field ^n ) ^{\hom (\field ^n , V)}
\]
is a monomorphism for all $V\in \fdvs$.
\end{defn}

\begin{rem}
There is an even space variant: 
the spectrum $E$ satisfies  the even height--$n$ detection
property if the above condition is satisfied for the functor  $E^{2
*} $.
\end{rem}

\begin{hyp}
\label{hyp:E-hkr}
Let $E$ be a complex-oriented ring spectrum which satisfies the
following conditions:
\begin{enumerate}
\item
the coefficient ring $E^*$ is a complete, local graded ring with graded
maximal ideal $\mathfrak{m}$;
\item
the graded residue field $\kappa \co= E^* / \mathfrak{m}$ has
characteristic $p>0$;
\item
the ring $p^{-1} E^*$ is nonzero;
\item
the induced formal group over $E^* / \mathfrak{m}$ has height $n$.
\end{enumerate}
\end{hyp}

\begin{prop}
\label{prop:hkr}
A complex-oriented ring spectrum  $E$ which satisfies \fullref{hyp:E-hkr} satisfies the height--$n$ detection property.
\end{prop}

\begin{proof}
See the proof of \cite[Proposition 3.6]{ghs}.
\end{proof}

\begin{exam}
\label{exam-cohom-theories}
The following theories satisfy the height--$n$ detection property:
\begin{enumerate}
\item
the $I_n$--adically complete Johnson--Wilson theory  $\widehat{E(n)}$;
\item
Morava's integral lift of $n$--th Morava $K$--theory, which has coefficient ring $W (\field _{p^n}) [u, u ^{-1}]$, where $ W
(\field _{p^n})$ denotes the Witt vectors;
\item
the Lubin--Tate theory $E_n$ with coefficient ring $$E_{n*}\cong
W (\field _{p^n}) [[w_1, \ldots , w_{n-1}]][u, u ^{-1}].$$
\end{enumerate}
\end{exam}

The fundamental consequence of the height--$n$ detection property is
given by the following result, which is a formal consequence of the hypothesis.

\begin{prop}
\label{prop:n-detect}
Suppose  that the spectrum $E$ satisfies
the height--$n$ detection property. Then the map
$$
\delta_{[-,X]_E}\co [BV, X]_E 
\rightarrow 
[B\field ^n, X]_E ^{\hom(\field^n, V)} 
$$ 
is injective, for $V \in \fdvs$  an elementary abelian
$p$--group.
\end{prop}

\begin{proof}
Straightforward.
\end{proof}

\subsection[Combining the Hurewicz and the height--n detection
properties]{Combining the Hurewicz and the height--$n$ detection
properties}

The above results combine to give the following Corollary; 
there is an  analogous result when restricting to the even spaces of
the $\Omega$--spectrum.

\begin{cor}
\llabel{cor-hurewicz-height}
Suppose that $E$ satisfies the
height--$n$ detection property and that the Hurewicz hypothesis
$\hurewicz (B\field ^n)$ is satisfied. Then  the map
$$
Q_E \co 
[BV, X]_E 
\rightarrow 
[BV,X]_{H_E}
$$
is bijective. 
\end{cor}

\begin{proof}
The map $Q_E$ is surjective by \fullref{prop:QE}, thus 
we must show that  $Q_E$ is injective; hence it is
sufficient to show that the composite 
$$
[BV, X] _E 
\stackrel{Q_E}{\longrightarrow}
[BV, X]_{H_E} 
\stackrel{\delta_{[-, X]_{H_E}}}{\longrightarrow}
[B\field ^n, X]_{H_E} ^{\hom(\field^n, V)} 
$$
is injective.

Consider the commutative diagram
$$
\xymatrix{
[BV, X] _E 
\ar[d]_{Q_E}
\ar@{^(->}[rr]^(.43){\delta_{[-, X]_{E}}}
&&
[B\field ^n, X]_E ^{\hom(\field^n, V)} 
\ar[d]^{(Q_E)^{\hom(\field^n, V)}} 
\\
[BV, X]_{H_E} 
\ar[rr]_(.43){\delta_{[-, X]_{H_E}}}
&&
[B\field ^n, X]_{H_E} ^{\hom(\field^n, V)}.
}
$$
The right hand vertical arrow is a bijection by  \fullref{prop:E-Hurewicz}, and the top map is injective by 
\fullref{prop:n-detect}. The result follows.
\end{proof}



\section{Filtrations of categories of presheaves}
\label{sect-filt}

This section considers natural ways of defining filtrations on certain
functor categories. This is presented for general categories of
presheaves  and is then specialized to the cases of interest, namely
certain natural filtrations of the functor category $\g$. The
definitions of this section are motivated by the height--$n$ detection
property introduced in \fullref{subsect:height-n}.

When applied to unstable algebras, the constructions introduce a
family of functors $\alpha \restrict \co \unstalg \rightarrow \unstalg$,
for $n \in \mathbb{N} \cup \{\infty \}$. These functors take values in
nilclosed unstable algebras; they are defined in
\fullref{subsect:alphan} and their basic properties are established
in \fullref{prop:alphan} and \fullref{thm:tdeg-alphan}.

\subsection{Generalities for presheaves}
Throughout this section, let $\calc$ denote an (essentially) small category. 

\begin{nota}$\phantom{9}$
\begin{enumerate}
\item
Let $\presh$ denote the full subcategory $\mathrm{Funct}(\calc^{\mathrm{op}}, \set)$ of \break presheaves.
\item
Let $\pfpresh$ denote the category of presheaves with values in the category $\pfset$ of profinite sets. 
\end{enumerate}
\end{nota}

For clarity of presentation, the following definitions are considered only for the category of presheaves $\presh$. 

\begin{defn}
Let $\cald \hookrightarrow \calc$ be a subcategory and let $F \in \presh$ be a presheaf.
\begin{enumerate}
\item
For an object $C \in \calc$, let $\sim _\cald$ denote the equivalence
relation on $F(C)$ defined by $x \sim _\cald y $ if
and only if $j ^* (x) = j^* (y)$  for each morphism $j \co D \rightarrow C$, where $D\in \cald$. 
\item
Let $\beta_\cald F (C) $ denote the quotient set $F(C) / {\sim_\cald}$.
\end{enumerate}
\end{defn}

\begin{lem} $\phantom{9}$
 \begin{enumerate}
\item
The association $C \mapsto \beta_\cald F(C)$ defines a presheaf
$\beta_\cald F$. 
\item
The association 
$\beta_\cald$ defines a functor $\beta_\cald \co  \presh \rightarrow \presh$.
\item
There is a natural surjection $1 _{\presh} \rightarrow \beta_\cald$. 
\item
For $D$ an object of $\cald$, $\beta_\cald F (D) = F(D)$. 
\item
The functor $\beta_\cald$ depends only upon the full subcategory generated by $\cald$. 
\end{enumerate}
\end{lem}
  
\begin{proof}
Straightforward.
\end{proof}

The following result has important consequences:

\begin{prop}
Let $\cald \hookrightarrow \calc$ be a full small subcategory, and let $\beta_\cald \co  \presh
\rightarrow \presh$ denote the induced functor. Then
\begin{enumerate}
\item
$\beta_\cald$ preserves surjections;
\item
$\beta_\cald$ preserves injections;
\item 
$\beta_\cald$ commutes with finite products: namely, for presheaves  $F, G \in \presh$, there is a canonical isomorphism $(\beta_\cald F) \times (\beta_\cald G) \cong \beta_\cald (F \times G)$. 
\end{enumerate}
\end{prop}

\begin{proof}
The first statement is a consequence of the definition of
$\beta_\cald$ as a natural quotient. The second statement is
essentially formal: suppose that $\gamma\co  F  \hookrightarrow G$ is an
inclusion of presheaves and let $x, y \in F(C) $ represent distinct
elements in $\beta_\cald F$. It is necessary to show that $\gamma_C x$ and $\gamma_C y$ represent distinct elements in $\beta_\cald G(C)$. By definition, there exists a morphism $f\co  D \rightarrow C$ with $D \in \cald$ for which $f^* (x) \neq f^* (y)$ in $F(D)$. Hence, the images $\gamma_D f^* (x)$ and $\gamma_D f^* (y)$ are distinct in $G (D)$, since $\gamma_D$ is an injection. Functoriality gives $\gamma_D f^* = f^* \gamma_C$, so the result follows. 

Let $F, G  \in \presh$ be presheaves and let $C$ be an object in $\calc$. The result follows from the observation that $(x_1, y_1) \sim_\cald (x_2, y_2)$ in $(F\times G)(C)$ if and only if $x_1 \sim_\cald x_2$ and $y_1 \sim_\cald y_2$.
\end{proof}

\subsection[A fundamental property of the constructed induced functor]{A fundamental property of the construction $\beta _\cald$}

The notation of the previous section is maintained.

\begin{prop}
\label{prop:fund-beta}
Let $\nu\co  F \rightarrow G$ be a morphism of presheaves in $\presh$
which factors as $\nu_\cald \co  \beta _\cald F \rightarrow G$. Then
$\nu_\cald$ is a monomorphism of presheaves if and only if the
restriction $F|_\cald \rightarrow G|_\cald$ is a monomorphism. 
\end{prop}

\begin{proof}
The forward implication is straightforward, hence consider the
converse. 
If $\nu_\cald \co  \beta_\cald F \rightarrow G$ is not a
monomorphism, then there exist sections $x, y \in F (C)$, for some
object $C$ of $\calc$ such that $x$ and $y$ represent distinct elements in
$\beta_\cald F (C) $ and $\nu (x) = \nu (y)$. The first condition implies
that there exists a morphism $j\co  D \rightarrow C$ in $\calc$, where
$D$ is an object of $\cald$, such that $j^* x \neq j^* y$ in $F (D) =
\beta _\cald F (D)$. It follows that the morphism $F (D) \rightarrow G
(D)$ is not a monomorphism, by naturality,  which implies the required result. 
\end{proof}

\subsection{Presheaves of profinite sets}

The general  application of these constructions to unstable modules
requires the consideration of presheaves with values in profinite
sets. (For the applications of this paper, it would be sufficient to
restrict to presheaves with values in finite sets).

\begin{prop}
The functor $\beta _\cald \co  \presh \rightarrow \presh$ induces a functor on the category of presheaves with values in profinite sets, $\beta_\cald \co  \pfpresh \rightarrow \pfpresh$. 
\end{prop}

\begin{proof}
Straightforward.
\end{proof}

\subsection{Presheaves on finite-dimensional vector spaces}
The motivating example is the category of presheaves of profinite sets
on the category of finite-dimensional vector spaces over the prime
field $\field$.

\begin{nota}
 For $n$ a nonnegative integer or $\infty$, let $\fdvs \restrict$
 denote the full subcategory of $\fdvs$ with set of objects $\{ \field ^t\, |\, t \leq n \}$.
\end{nota}

Recall that $\g$ denotes the category of functors from $\fdvs$ to
$\pfset \op$, so that $\g \op$ is equivalent to a category of
presheaves on $\fdvs$ with values in profinite sets. The general definitions for presheaves yield the following functors, in which attention must be paid to the  variance.

\begin{defn}
For $n$ a nonnegative  integer or $\infty$, let $\beta \restrict \co  \g \op  \rightarrow \g \op$ denote the functor
$\beta_{\fdvs \restrict}$ defined on the category of presheaves on
$\fdvs$ with values in profinite sets.
\end{defn}

By construction,  there is a canonical surjection
in $\g \op$,   $G \twoheadrightarrow \beta \restrict G$ for each $G \in \g$.

\begin{prop}
\label{prop:beta-prop}
For $n$ a nonnegative  integer:
\begin{enumerate}
\item
the functor $\beta \restrict[\infty]$ is equivalent to the identity
functor; 
\item 
there is a natural surjection 
 $ \beta \restrict[n+1]  \twoheadrightarrow \beta \restrict[n]$ of
 endofunctors on $\g \op$ which occurs in the diagram of natural
 transformations 
$$
\xymatrix{
1= \beta \restrict[\infty]
\ar[r]
\ar[rd]
&
\beta \restrict[n+1]
\ar[d]
\\
&
\beta \restrict.
}
$$
\end{enumerate}
\end{prop}

\begin{proof}
The first statement is clear; to prove the second, it is sufficient to
show that the natural transformation $1 \rightarrow \beta \restrict$
factorizes naturally across $\beta \restrict[n+1]$. 

Let $G$ be an object of $\g \op$ and consider sections $s_0, s_1 \in G
(V)$ such that, for any morphism $j\co  \field^{n+1} \rightarrow V$, the
sections $j^* s_0$ and $j^*s_1$ are equal in $G (\field^{n+1})$. Consider
a morphism $k\co  \field ^n \rightarrow V$; there exists a factorization
$\field^n \rightarrow \field^{n+1} \stackrel{j}{\rightarrow} V$ across
a morphism $j$. It follows that $k^* s_0 = k^* s_1$, which  implies that the morphism $1 \twoheadrightarrow \beta \restrict$ factors across $1\twoheadrightarrow \beta \restrict[n+1]$, as required.
\end{proof}

\begin{exam}
\llabel{exam-alpha-BV}
Let $\hom (-,W)$ denote the object of $\g \op$ given by $V
\mapsto \hom(V, W)$ (forgetting the abelian structure) and let $n$ be
a nonnegative integer. Linearity implies that $\beta \restrict \hom
 (-, W) \cong \hom (-,W)$, so that the  filtration of $\hom (- ,W)$
 provided by $\beta \restrict $ as $n$ varies is constant.
\end{exam}

\begin{prop}
\label{prop:beta-analytic}
For $n$  a nonnegative integer or $\infty$,  the functor $\beta \restrict $
restricts to  a functor $\beta \restrict \co  \g_\omega \op \rightarrow
\g_\omega \op$.
\end{prop}

\begin{proof}
Consider an object $G \in \g \op$ which corresponds to an analytic
functor. 
The canonical surjection $G \twoheadrightarrow \beta \restrict G$
induces a monomorphism in $\f$ under the functor $\mapc (-,
\field)$. This identifies $\mapc (\beta \restrict G, \field)$ as a
subobject of an analytic functor by the hypothesis on $G$, hence
$\beta \restrict G$ is analytic, as required.
\end{proof}

\subsection[Functors induced from finite-dimensional vector spaces over finite fields]{Functors induced from $\fdvs \restrict[d]$}
The theory of Henn, Lannes and Schwartz \cite{hls} provides an elegant treatment of unstable
algebras of transcendence degree $d$,  for $d$ a nonnegative integer. This section shows that the
functors $\beta \restrict$ restrict well to this theory. 

\begin{nota}
For $d$ a nonnegative integer, let $\g \restrict[d]$ denote the category of functors from $\fdvs
\restrict [d]$ to $\pfset \op$, so that $\g \restrict[d] \op$ is the
category of presheaves on $\fdvs \restrict[d]$ with values in
profinite sets. Let $e_d \co  \g \op \rightarrow \g \restrict[d]\op$ denote the
canonical restriction functor. 
\end{nota}

The importance of the category $\g \restrict[d]$ is given by its
connection with the category $\unstalg$ \cite{hls}. Let
$\unstalg_d$ denote the full subcategory of $\unstalg$ the  objects
of which  have
underlying algebra of transcendence degree at most $d$;  the  category $\unstalg_d$
is closed under $F$--isomorphism and the following result holds.

\begin{thm}
{\rm \cite[Theorem II.2.8]{hls}}\qua
The functor $g\co  \unstalg \rightarrow \g$ induces an equivalence of
categories
 $
\unstalg_d / \nil
\cong
\g \restrict[d].
 $
\end{thm}

The category $\g \restrict[d]\op$ is equivalent to the category of
profinite right $\mathrm{End} (\field ^d)$--sets
\cite[Section II.2]{hls}. The functor $e_d$ identifies with the
evaluation functor on $\field ^d$ and admits a left adjoint  
$
i_d \co  \g \restrict[d] \op \rightarrow \g \op
 $ 
which is given by $S \mapsto S (\field ^d)\times _{\mathrm{End}
(\field^d) } \hom (- , \field ^d) $. A fundamental property is  the following result.  
\begin{lem}
\llabel{lem-key-hls}
{\rm \cite[Key Lemma, Lemma 2.1]{hls}}\qua
Let $G \in \g \op$ be a presheaf of profinite sets on $\fdvs$. Then
the canonical morphism $i_d e_d G \rightarrow G$ is a monomorphism.
\end{lem}

\begin{nota}
Let $\beta \restrict \co  \g \restrict[d] \op \rightarrow \g
\restrict[d] \op$ denote the analogue of the functor \break$\beta \restrict \co 
\g \op \rightarrow \g \op$. 
\end{nota}

The following result gives a criterion for the natural transformation
$1 \rightarrow \beta \restrict$ to yield an isomorphism when evaluated on an object.

\begin{prop}
\llabel{prop-Gd-alpha}
For nonnegative integers $n$ and $d$  and  $S \in \g \restrict[d]\op$
 a profinite presheaf on $\fdvs \restrict[d]$, there is a
natural isomorphism
\[
i_d (\beta \restrict S) 
\stackrel{\cong}{\rightarrow}
\beta \restrict (i_d S).
\]
In particular for $n \geq d$, there is an identification $\beta
\restrict (i_d S) =i_d S$.
\end{prop}

\begin{proof}
Write $G$ for $i_d S$, then there is a natural surjection $G
\twoheadrightarrow \beta \restrict G$. From the definitions, it is
clear that there is an identification $e_d (\beta \restrict G) =
\beta \restrict (e_d G)$. The composite 
\[
i_d \beta \restrict (e_d G) 
\stackrel{\cong}{\rightarrow}   
i_d e_d (\beta \restrict G) 
\hookrightarrow \beta \restrict G
\]
 is a monomorphism by \fullref{lem-key-hls}. This implies the
 first  statement; the second follows immediately.
\end{proof}

\subsection[The alpha functors and unstable algebras]{The functor $\alpha \restrict$ and unstable algebras}
\label{subsect:alphan}

The functor $\beta \restrict \op$ induces a functor on the
category $\unstalg / \nil$ by \fullref{thm:hls}; composing  with the functors involved in
the nil-localization $\unstalg \leftrightarrows \unstalg/ \nil$
yields a functor on the category $\unstalg$ of unstable algebras. (Note that the variance considerations for the functor $g$ require passage to the opposite category.) The
reader is referred to Henn, Lannes and Schwartz \cite{hls} for the relevant definitions
concerning the nil-localization of unstable algebras.

\begin{nota}
For $K \in \unstalg$ an unstable algebra and  $n$ a nonnegative integer or $\infty$,  write $\alpha \restrict
K$ for the nilclosed unstable algebra which is associated to the functor $\beta
\restrict \op g K$.
\end{nota}

The properties of the functor $\beta \restrict \co \g \op \rightarrow \g \op $ imply the following properties of the induced functor on unstable algebras.

\begin{prop}
\label{prop:alphan}
For $n$ a nonnegative integer or $\infty$  and  $K\in \unstalg$  an unstable
algebra, the following properties hold. 
\begin{enumerate} 
\item
The association $\alpha \restrict \co  K \mapsto \alpha \restrict K$
defines a functor $\alpha \restrict \co  \unstalg \rightarrow
\unstalg$ with values in nilclosed unstable algebras.
\item
The functor $\alpha \restrict[\infty]$ identifies with the nilclosure functor.
\item
 The functor $\alpha \restrict \co  \unstalg \rightarrow \unstalg $
 preserves injections.
\item
There is a  commutative diagram of natural monomorphisms 
$$
\xymatrix{
\alpha \restrict[n] K
\ar@{^(->}[rd]
\ar@{^(->}[d]
\\
\alpha \restrict[n+1] K
\ar@{^(->}[r]
&
 \alpha \restrict[\infty] K.
}
$$
\end{enumerate}
\end{prop}

\begin{proof}
The result is a direct translation of \fullref{prop:beta-prop} using \fullref{prop:beta-analytic} to apply \fullref{thm:hls}.
\end{proof}

The behaviour of the functors $\alpha \restrict $ on unstable algebras
of finite transcendence degree is described by the following result.

\begin{thm}
\label{thm:tdeg-alphan}
If $K \in \unstalg_d $ is an unstable algebra of transcendence degree
$d$ then $\alpha \restrict K$ is an unstable algebra of transcendence
degree at most $d$. Moreover, if $n \geq d$, then 
$
 \alpha \restrict K
 $ 
identifies with the nilclosure of $K$.
\end{thm}

\begin{proof}
The result is an immediate consequence of \fullref{prop-Gd-alpha}.
\end{proof}



\section[On the varieties associated to the E--cohomology rings]{On the varieties associated to the rings $H^*_E(-)$}
\label{Sec-varieties}

The study of the algebraic variety associated to a ring $H^*_E (X)$ is
essentially equivalent to the study of the functor $g H^* _ E (X)
\in \g$. The results of this section are stated for the rings $H^*
_E$; there are analogous results for the rings $\chern ^*$. 

\subsection{Preliminaries}
To generalize the results of Green, Hunton and Schuster \cite{ghs} requires realizing
morphisms between certain unstable  algebras topologically.  This uses
part of the theory of unstable algebras.

Recall the nonlinear injectivity property
of the objects $H^* (BV) \in \unstalg$. 

\begin{prop}
\label{prop:non-lin-inj}
{\rm \cite[Corollary 3.8.7]{s}}\qua 
For a monomorphism $i\co  K \hookrightarrow L$ and  a morphism $\phi \co  K
\rightarrow H^* (BV)$ in $\unstalg$,  there
exists an extension $\tilde \phi \co  L \rightarrow H^* (BV)$ in
$\unstalg$ such that $\phi = \tilde \phi \circ i$.
\end{prop}

\begin{cor}
\llabel{cor-top-realization}
Let $X$ be a topological space for which the Hurewicz morphism $[BV, X] \rightarrow
\hom_{\unstalg} (H^* (X) , H ^* (BV)) $ is surjective, where $V \in
\fdvs$ is an elementary abelian $p$--group.
Then there are canonical bijections: 
\begin{align*}
[BV, X] _H  &\cong  \hom_{\unstalg} (H^* (X), H^* (B V));
\\
[BV , X]_{H_E}  &\cong  \hom_{\unstalg} (H^*_E (X), H^*_E (BV)).
\end{align*}
\end{cor}

\begin{proof}
The first bijection is a special case of \fullref{lem:Hurewicz-surj}. To prove the second, it is sufficient to show
that the functor $H^* _E$ induces a surjection from $[BV  ,X]$ onto $\hom_{\unstalg} (H^* _E (X),H^* _E (BV))$. 

A  morphism 
$\alpha\co  H^* _E
(X) \rightarrow H^* _E(BV)$ of unstable algebras extends to a morphism $H^* (X) \rightarrow H^* (BV)$ by \fullref{prop:non-lin-inj}. Hence, $\alpha$ is in the image of the morphism $[BV, X] _H
\rightarrow [BV, X]_{H_E}$ and the result follows by the first
statement. 
\end{proof}

The hypothesis on the surjectivity of the Hurewicz homomorphism is
satisfied, for example, if the topological space $X$ satisfies the hypotheses of
the following result.

\begin{thm}{\rm \cite{l}}\qua
\llabel{thm-Miller-conj}
Let $X$ be a connected topological spaces which is of finite type, nilpotent and such that $\pi_1 (X)$ is finite. Then the Hurewicz morphism
\[
[BV, X] \rightarrow
\hom_{\unstalg} (H^* (X) , H ^* (BV)) 
\]
 is an isomorphism. 
\end{thm}

\subsection[Algebraic approximations]{Algebraic approximation to the rings $H^* _E$}

 The theory of the Chern subring for group cohomology \cite{ghs} admits a partial
generalization under suitable hypotheses; the general result is stated
below in \fullref{thm-comparison}.

\begin{thm}
\llabel{thm-comparison}
Suppose that there exists a positive integer
$n$ such that 
\begin{enumerate}
\item
the Hurewicz hypothesis $\hurewicz (B\field ^n)$ holds;
\item
$E$ satisfies the height--$n$ detection property.
\end{enumerate}
Then for $X$ a topological space for which the  morphism 
$$[BV, X] \rightarrow
\hom_{\unstalg} (H^* (X) , H ^* (BV)) 
$$
 is surjective, 
  the natural transformation $1 \rightarrow \beta \restrict$ induces an isomorphism of functors in $\g \op$
\[
g\op H^* _E (X) 
\stackrel{\cong}{\rightarrow}
\beta \restrict g\op H^* _E (X) .
\]
\end{thm}
 
\begin{proof}
The natural transformation $1 \rightarrow \beta \restrict$ is
surjective, hence it is sufficient to prove injectivity. Let $V\in \fdvs$ be an elementary abelian $p$--group and let  $y_1, y_2 \in g
H^*_E (X) (V)$ be elements such that  $j^* y_1= j^* y_2$, for all morphisms $j\co  \field ^n \rightarrow
V$.  The definition
of the functor $g\co  \unstalg \op \rightarrow \g \op$ together with
\fullref{cor-top-realization} imply that the elements $y_i$ are
represented by morphisms $H^* _E (f_i)$ for $f_i \in [BV, X]$, $i \in
\{1, 2\}$. 

The hypothesis upon the elements $y_i$ implies that the two 
composites in the following  diagram coincide:
\[
\xymatrix{
H^* _E (X) 
\ar@<1ex>[r]^{f^*_1} 
\ar@<0ex>[r]_{f^*_2} 
&
H^* _ E (BV) 
\ar[r] ^{j ^*}
&
H^*_E (B \field ^n),
}
\]
for all morphisms $j\co  \field ^n \rightarrow V$. The morphisms $E^* (f_1), E^*
(f_2) \co  E^* (X) \rightrightarrows E^* (BV)$ coincide by \fullref{cor-hurewicz-height}; hence
\fullref{prop:E-Hurewicz} implies that the elements $y_1$ and $y_2$ are equal. It follows that the canonical morphism $g\op H^* _E (X) 
\twoheadrightarrow
\beta \restrict g\op H^* _E (X)$ is a monomorphism, as required.
\end{proof}

\begin{cor}
\llabel{cor-canon-F-mono}
Let $E$ be a spectrum which satisfies the hypotheses of \fullref{thm-comparison}, and let $X$ be a topological space
 for which the  morphism $$[BV, X] \rightarrow
\hom_{\unstalg} (H^* (X) , H ^* (BV))$$ is surjective. 
 Then there is a canonical morphism 
\[
H^* _E (X) 
\rightarrow 
\alpha \restrict H^* (X) 
\]
in $\unstalg$ which induces a monomorphism in $\unstalg / \nil$.  
\end{cor}

\begin{proof}
Recall that the functor $\alpha \restrict[\infty]$ coincides with the
nil-localization functor and that there is a canonical morphism $H^*_E
(X) \rightarrow \alpha \restrict[\infty]H^*_E
(X)$ which induces an isomorphism in $\unstalg / \nil$. There is a commutative diagram 
$$
\xymatrix{
& 
H^*_E (X) 
\ar[d] \\
\alpha \restrict H^* _E (X) 
\ar[r]^\cong 
\ar@{^(->}[d]
&
\alpha \restrict [\infty]H^* _E (X) 
\ar@{^(->}[d]
\\
\alpha \restrict H^* (X) 
\ar[r]
&
\alpha \restrict [\infty]H^* (X), 
}
$$
 in which the top horizontal morphism is an isomorphism by \fullref{thm-comparison}. The left hand vertical
 morphism  is a monomorphism, since  $\alpha \restrict \co  \unstalg \rightarrow
\unstalg$ preserves injections. The result follows.
\end{proof}

\begin{rem}$\phantom{9}$
\begin{enumerate}
\item
In general, the
morphism $H^* _E (X) 
\rightarrow 
\alpha \restrict H^* (X)$ is not a monomorphism, since
 the unstable algebra $H^*_E (X)$ need not be nilclosed. For example, consider
the case $E= KU$, unitary $K$--theory and $X = \mathbb{C} P^n$. 
\item
 It is not clear  that the topological functor $X \mapsto H^* _E (X)$ extends in
general to a functor $\unstalg \rightarrow \unstalg / \nil$,
especially in the light of  \fullref{exam:proper-incl} below.
\end{enumerate}
\end{rem}

\subsection{Separation}

It is natural to seek a class of spaces for which the natural morphism from 
$H^* _E (X)$ to $\alpha \restrict H^* (X)$ is an isomorphism in
$\unstalg / \nil$. For example, the results of Green, Hunton and Schuster \cite{ghs}
concern classifying spaces of finite groups; in this case,
Hopkins, Kuhn and Ravenel \cite{hkr} provide a construction of enough elements in
$E^*$--cohomology so as  to obtain equality up to $F$--isomorphism.

\begin{lem}
\label{lem:E-separation}
Suppose that there exists a positive integer
$n$ such that the Hurewicz hypothesis $\hurewicz (B\field ^n)$ holds.
Then there exists a commutative diagram 
 $$
\xymatrix{
[B\field ^n, X] \ar@{->>}[r]
\ar@{->>}[d]
&
[B\field ^n, X]_E
\\
[B \field ^n , X]_H 
\ar@{->>}[ur]_{P_{H,E}}
}
$$
of surjections. 
\end{lem}

\begin{proof}
The hypothesis implies that  $Q_E\co  [B \field^n, X] _E
\rightarrow [B\field^n, X] _{H_E}$ is an isomorphism by \fullref{prop:E-Hurewicz}. The result follows from the commutative diagram 
$$
\xymatrix{
[B \field ^n , X] \ar[r]
\ar[d]
&
[B \field ^n , X]_E
\ar[d]^{Q_E}_\cong 
\\
[B \field ^n , X]_H 
\ar[r]
\ar@{.>}[ur]^{P_{H,E}}
&
[B \field ^n , X]_{H_E}.
}
$$
\proved
\end{proof}

The following result gives a separation criterion which is sufficient to prove equality in $\unstalg/ \nil$ between the unstable algebras $H^* _E (X)$ and $\alpha \restrict H^* (X)$, under suitable hypotheses.

\begin{prop}
\label{prop:E-separation}
Suppose that there exists a positive integer
$n$ such that 
\begin{enumerate}
\item
the Hurewicz hypothesis $\hurewicz (B\field ^n)$ holds;
\item
$E$ satisfies the height--$n$ detection property.
\end{enumerate}
Let $X \in \topspace$ be a space for which the  morphism 
$$[B\field ^n, X] \rightarrow
\hom_{\unstalg} (H^* (X) , H ^* (B\field ^n)) $$
 is surjective.  

Suppose that the morphism 
$$
P_{H,E} \co  [B\field ^n, X]_H \rightarrow [B \field ^n ,X]_E
$$
 is an injection. Then there is an isomorphism $H^* _E (X)\cong _{\unstalg/\nil} \alpha \restrict H^* (X)$ in $\unstalg/ \nil$.
\end{prop}

\begin{proof}
The morphism $H^* _E (X) \rightarrow \alpha \restrict H^* (X)$
corresponds to a morphism 
$$
\beta \restrict g \op H^* (X) \rightarrow g
\op H^*_ E (X),$$
and it is sufficient to show that this is a
monomorphism under the given hypothesis.  It is sufficient to show that the
morphism $g \op H ^* (X) \rightarrow g \op H^* _E (X) $ is a
monomorphism when evaluated on $\field ^n$ by \fullref{prop:fund-beta}.

The hypothesis upon $P_{H,E}$ implies that the morphism 
$
[B \field ^n, X]_H  \rightarrow 
[B \field ^n ,X] _{H_E} 
$
is an isomorphism. Moreover, the surjectivity of the Hurewicz
monomorphism implies that there is an isomorphism 
$ [B \field ^n, X]_H \cong \hom_{\unstalg} (H^* (X), H^* (B \field
^n))$ and an isomorphism $[B \field ^n, X]_{H_E} \cong \hom_{\unstalg} (H^*_E (X), H^*_E (B \field
^n))$. Hence the natural morphism 
$$
g \op H^* (X)  (\field ^n) \cong g \op H^*_E (X)(\field ^n)
$$ 
is an isomorphism, as required.
\end{proof}

\subsection[An example for E=KU]{An example for $E=KU$}

The following example  shows that  the monomorphism in $\unstalg /
\nil$ provided by \fullref{cor-canon-F-mono}  can be strict.

\begin{exam}
\label{exam:proper-incl}
Consider the Eilenberg-MacLane space $K (\field, 2) $; the results of Anderson
and Hodgkin imply that $K^* ( K (\field, 2))$ identifies with $K^*
(pt)$ \cite{ah}. Theorem 4.6 of \cite{hs} implies that the
subalgebras $$H^*_{\tiny\widehat{E(1)}} (X) \quad \text{and} \quad H^* _{KU} (X)$$ of $H^*
(X)$ coincide, hence \fullref{cor-canon-F-mono} applies  and there is a  monomorphism in $\unstalg / \nil$
$$
H^*
_{KU}(K(\field, 2))
\hookrightarrow 
\alpha \restrict[1] 
H^* (K (\field, 2)). $$
This morphism does not induce an isomorphism in $\unstalg / \nil$.
To see this,  it is sufficient to show that $\beta \restrict[1]  g \op H^* (K (\field, 2))$
is not trivial; this is straightforward. 
\end{exam}

\begin{rem}
The Eilenberg--MacLane spaces provide universal  examples  of the failure of the
$F$--monomorphism to be an $F$--isomorphism. It is this failure which provides an obstruction to extending the functor $H^* _E$ to an algebraic functor defined on $\unstalg$.
\end{rem}



\section{Filtrations of functors and group cohomology}
\label{Sec-filt-group-cohom}
\reset

This section reviews the results of \cite{ghs} from the functorial
point of view. This functorial presentation allows the generalization
of these results to the case
of Borel cohomology, which is presented in \fullref{sect-borel}.

\begin{nota}
 Throughout this section, let $G$ be a finite group and let $n$ be a nonnegative integer.
\end{nota}

\subsection{Quillen's theorem -- the functorial approach}

This section gives an overview of  the reinterpretation, due to Lannes, of
Quillen's results on the algebraic variety associated to the  cohomology of a finite
group. For convenience, references are given in \cite{s} rather than
in the original sources. A further convenient reference is the section
on group cohomology of \cite{dh}.

Recall the definition of the Quillen  category  associated to a finite group $G$. 

\begin{defn}
Let $\cala _p (G)$ denote the category with objects the elementary
abelian $p$--subgroups of $G$ and  morphisms the morphisms  of groups which are composites of
conjugation by an element of $G$ and the inclusion of subgroups.
\end{defn}

\begin{nota}
Let $\calc_p ^G$ denote the set of presheaves on the category $\cala_p
(G)$, namely functors $\cala_p (G) \op \rightarrow \set$.
\end{nota}

There is a forgetful functor $\cala_p (G) \rightarrow \fdvs$, hence
 the set of homomorphisms between vector spaces induces a
 functor $\hom ( -, -) \co  \fdvs \op \times \cala_p (G) \rightarrow
 \set$. Let $F\in \calc _p ^G$ be a presheaf, then the product $F \times \hom (-, -)$
 induces a functor 
 $$\cala_p (G ) \op \times \cala_ p (G) 
\rightarrow \presh[\fdvs],$$
 where $\presh[\fdvs]$ denotes the category of functors $\fdvs \op
 \rightarrow \set$. If $F$ takes values in the category of finite
 sets, this defines an object of $\g \op$.  

\begin{defn}
For $F \in \calc_p^G $ a functor which takes values in finite
sets, let  $\int_{\cala_p (G)} F $ be the functor in $\g
\op$ which is the coend of the functor $F \times \hom (-, -)$ \cite[Chapter IX, Section 6]{mcl}. Explicitly, this is the quotient
of the functor 
\[
V \mapsto \amalg _{E \in \cala_p (G) } F (E) \times  \hom (V, E)
\]
by the quotient relation which is induced for morphisms $j\co  E_1
\rightarrow E_2$ in $\cala_p (G)$ by identifying the images of the
morphisms
\[
\xymatrix{
& F(E_2) \times  \hom(V, E_1 )  
\ar[ld] _{1 \times j_*}
\ar[rd] ^{j^*\times 1}
\\
 F(E_2) \times  \hom(V, E_2 )  
&&
 F(E_1) \times  \hom(V, E_1 ).  
}
\]
\end{defn}

\begin{exam}$\phantom{9}$
\label{exam:Rep-FX}
\begin{enumerate}
\item
Let $F\in \calc_p^G $ be the constant functor sending each $E\in
\cala_p (G)$ to a singleton set. The associated functor $\int _{\cala
_p (G) } F $ identifies with the co\-limit 
$$\varinjlim_
{E \in \cala _p (G)} \hom
(- , E).$$
 This functor is isomorphic to the functor $V \mapsto \rep
(V, G) $, where  $\rep (V, G)$ is  the set of group  homomorphisms modulo the
equivalence relation induced by conjugation by elements of $G$ \cite[Corollary 3.8.8, Theorem 3.10.2]{s}. 
\item
For $X$ a finite $G$--CW--complex, let $F_X$ denote the functor sending $E \in \cala_p (G)$ to the set $\pi_0 (X^E) / C_G (E)$, where
$C_G (E)$ denotes the centralizer of $E$ in $G$. The associated
functor $\int _{\cala _p (G) } F_ X$ evaluated on $V\in \fdvs$
gives a set of equivalence classes with representatives of the form $(a , x) $, where $a$ is a morphism $V
\rightarrow G$ and $x\in \pi_0 (X ^{\mathrm{Im} (a)})$.
\end{enumerate}
\end{exam}

Quillen's theorem on the variety of the ring $H^* (BG)$ has the
following interpretation in the context of unstable algebras.

\begin{thm}
\llabel{thm-Quillen}
For $G$  a finite group and  $X$  a finite
$G$--CW--complex, there are isomorphisms of functors in $\g \op$:
\begin{eqnarray*}
g\op H^* (BG)
&\cong& 
\rep (-, G);
\\
g\op H^* (EG \times_ G X) 
&\cong& 
\int _{\cala _p (G) } F_X.
\end{eqnarray*}
\end{thm}

The second statement is equivalent to the
description in terms of a limit over the Quillen category given in \fullref{prop:End-desc}. 

\begin{defn}
For $F \in \calc _p ^G$  a functor $\cala_p (G) \op \rightarrow \set$
which takes values in finite sets,  let $K (F) \in \unstalg$ denote the
unstable algebra over the Steenrod algebra given by
\[
K(F) \co =  \hom_{\calc_p ^G} (F, H^* (B-))
,
\]
where $H^* (B-)\co  E \mapsto H^* (BE)$ is regarded as a
functor by forgetting the unstable algebra structure. (This
definition is equivalent to defining a suitable end \cite[Chapter IX,
Section 5]{mcl}). The structure of an unstable algebra is induced by
the unstable algebra structure of $H^* (BE)$. 
\end{defn}

\begin{prop}
\label{prop:End-desc}
For $F  \in \calc _p^G$  a functor $\cala_p (G) \op \rightarrow \set$
which takes values in finite sets, there is an isomorphism of
functors in $\g \op$
\[
g\op K (F) 
\cong 
\int_{\cala_p(G)} F.
\]
\end{prop}

\begin{proof}
This is a basic exercise using the definition of the functor
$g\co  \unstalg \rightarrow \g$, together with the definition of the
respective end and coend. (See \cite[Corollary 3.8.8]{s}).
\end{proof}

\subsection{Chern subrings for group cohomology}

Theorem 0.1 of \cite{ghs} and Corollary 4.2 of \cite{hs}
 together imply the following result, in which $E(n)$ denotes the
 height--$n$ Johnson--Wilson spectrum.

\begin{thm}
\llabel{thm-group-cohom}
 For $G$ a finite group and  $n$  a nonnegative integer, the
 unstable algebra $\chern[E(n)]^*  (BG) $ has associated functor  in $\g \op$
\[
g\op \chern[E(n)]^* (BG)
\cong 
\beta \restrict 
\rep (-, G).
\]
\end{thm}

\begin{proof}
(Indications) \fullref{thm-comparison} implies that there is a
surjection  in $\g \op$
$$g\op \chern[\widehat{E(n)}]^* (BG)\twoheadrightarrow
\beta \restrict \rep (-, G),
$$
and this surjection is an isomorphism \cite{hkr}. The remainder of the proof  consists of the 
comparison between $g\op \chern[\widehat{E(n)}]^* (BG)$ and $g\op
\chern[{E(n)}]^* (BG)$ \cite{hs}.
\end{proof}

\subsection{Interpretation in terms of categories}
\label{sect:cat-interpret}

The following category, introduced by Green and Leary \cite{gl},  generalizes the  category $\cala_p(G)$ of
elementary abelian $p$--subgroups of $G$. 

\begin{defn}
Let $\acat (G)$ denote the category with the same objects as $\cala_p
(G)$  and morphisms the inclusions $V \hookrightarrow W$
such that, for each morphism $\field ^n \rightarrow V$, the restricted
morphism between subgroups of $G$  is induced by conjugation by an
element of $G$.
\end{defn}

The following Proposition is a generalization of the relation between
the colimit over the Quillen category $\cala_p (G)$ of $\hom (-, W)$  and the functor $\rep (-, G) $:

\begin{prop}
\llabel{prop-colimit}
For $G$  a finite group and  $n \geq 0$ an integer, there is a
natural isomorphism in $\g \op$
\[\varinjlim_{W \in \acat (G)}
\hom (-, W) 
\cong 
\beta \restrict
\rep (-, G).  
\]
\end{prop}

\begin{proof}
(Indications) There is an embedding of  categories $\cala_p (G)
\hookrightarrow \acat (G)$ and the categories have the same objects,
hence there is a surjection 
$$\rep (-, G) \twoheadrightarrow \lim_{W \in
  \acat (G)} \hom (-, W).
$$
 It is straightforward to verify that this
factors across the canonical surjection from $\rep (-, G)$ onto $\beta
\restrict \rep (-, G)$. 

In the other direction, for an object $W$ of $\acat (G)$, there is an
induced morphism $\hom (-, W) \rightarrow \rep (-, G) \rightarrow
\beta \restrict \rep (-, G)$. This morphism induces a morphism  $\lim_{W \in
  \acat (G)} \hom (-, W) \rightarrow \beta \restrict \rep (-, G)$. 

These morphisms induce the required isomorphism of functors. 
\end{proof}

This  recovers  the main result of \cite{ghs}:

\begin{thm}
For $G$ a finite group and $n \geq 0$ an integer, there exists a
natural $F$--isomorphism
\[
\chern[E(n)]^* (BG) 
\rightarrow 
 \varprojlim_{W \in  \acat (G)} H^* (BW).
\]
 \end{thm}



\section{Borel cohomology}
\label{sect-borel}

Throughout this section, let $G$ be a finite group, let $X$ be a
finite $G$--CW--complex, and let $n$ denote a fixed nonnegative integer. For clarity, the prime $p$ is suppressed from the notation. 

\subsection{Generalized Chern rings for equivariant cohomology}
The arguments of Green, Hunton and Schuster \cite{ghs} generalize to prove the following
Theorem, using the generalized character theory of Hopkins, Kuhn and Ravenel
\cite{hkr}. There
is an analogous result for the Chern subring, which is not stated here
explicitly.

\begin{thm}
\llabel{thm-equivariant-En}
Let $E$ be a complex-oriented theory which satisfies
\fullref{hyp:E-hkr} and the Hurewicz hypothesis $\hurewicz[E]
(B\field^n_p)$, for  a nonnegative integer $n$.  Then the unstable algebra  $H^*_E (EG\times _G X) $ is $F$--isomorphic to 
$\alpha \restrict H^* (EG\times_ G X)$.
\end{thm} 

\begin{proof}
The proof is an application of \fullref{prop:E-separation} to the
space $EG\times _G X$, which  involves the following steps:
\begin{enumerate}
\item
The Hurewicz morphism 
$$[B \field ^n , EG \times _ G X ] 
\rightarrow \hom_{\unstalg} (H^* (EG \times _G X), H^* (B \field
^n))$$
is surjective;
\item
The morphism which corresponds to $P_{H,E}$, 
 $$
g\op H^* (EG \times _G X) (\field^n) \cong [B \field ^n, EG \times _GX]_ H 
\rightarrow 
[B \field ^n, EG \times _GX]_ E,
$$
 is injective. 
\end{enumerate}
These statements are proved in \fullref{subs:surjectivity} and \fullref{subs:Separation}.
\end{proof}

This provides the following generalization of  \fullref{thm-group-cohom} with respect to $E(n)$. 

\begin{cor}
The unstable algebra  $\chern[E(n)]^*  (EG\times _G X) $ is
 $F$--isomorphic to \break
$\alpha \restrict H^* (EG\times_ G X)$, for $n$ a nonnegative integer. 
\end{cor}

\begin{proof}
\fullref{thm-equivariant-En} applies to the complete
Johnson--Wilson theory $\widehat{E(n)}$, by the results of \cite{ghs}
stated here in \fullref{exam:Kn}.  Then  Theorem 4.1 of
\cite{hs} applies to replace the complete Johnson--Wilson theory
by  $E(n)$. 
\end{proof}

\subsection{Surjectivity of the Hurewicz morphism} \label{subs:surjectivity}This section
establishes that the Hurewicz morphism 
$$
[B \field ^n , EG \times _ G X ] 
\rightarrow \hom_{\unstalg} (H^* (EG \times _G X), H^* (B \field
^n))$$
 is surjective and gives an explicit model for the right hand side.

\begin{defn}
Let $\mathcal{Y}_{n} $ denote the left $G$--, right $\mathrm{Aut}(\field ^n)$--set  with elements pairs $(\alpha , [x])$, where $\alpha \co  \field ^n  \rightarrow G$ is a group homomorphism and $[x ] \in \pi_ 0 (X^{\mathrm{Im}(\alpha)})$. The left $G$--action is via conjugation on $\hom (\field ^n  , G) $ and is  induced via the restricted action $g\co  X^{\mathrm{Im}(\alpha)} \rightarrow  X^{\mathrm{Im}(g \alpha g^{-1})}$; the right $\mathrm{Aut}(\field ^n)$--action is via the canonical action upon $\hom (\field ^n  , G) $.
\end{defn}

\begin{lem}
There is a map of sets $\mathcal{Y}_n \rightarrow [B \field ^n , EG
\times _G X] $ which sends a pair  $(\alpha , [x])$ to the homotopy
class of  a morphism
$$
B \field^n \rightarrow B \field^n \times X^{\mathrm{Im}(\alpha)} \rightarrow EG \times _G X,
$$
where the first morphism is induced by $[x]$ and the second morphism
is induced by $\alpha$. 
\end{lem}

\begin{proof}
This is a standard geometric construction (see \cite[Section 6]{hkr}).
\end{proof}

The set $\mathcal{Y}_n$  is related to the functor $F_X$ defined in \fullref{exam:Rep-FX} by the following result.

\begin{prop}
\label{prop:gop-model}
There is a canonical isomorphism of right $\mathrm{Aut}(\field ^n)$--sets
$$
\Big\{\int _{\cala_p  (G) } F_ X \Big\}
(\field ^n ) 
\cong 
\mathcal{Y}_{n}/ G.
$$
\end{prop}

\begin{proof}
An object of $\cala_p (G)$ is an elementary abelian $p$--group $V$ equipped with a monomorphism $V \hookrightarrow G$. There is a map $F_X (V) \times \hom(\field ^n  , V) \rightarrow \mathcal{Y} _{n} $ which is induced by composition with $V\hookrightarrow G$. It is straightforward to verify that this induces a surjective morphism
$$
\Big\{
\int _{\cala _p (G) } F_ X
\Big\}
(\field ^n ) 
\twoheadrightarrow
\mathcal{Y}_{n} / G
$$
and that this is a monomorphism.
\end{proof}

\begin{cor}
\label{cor:Hurw-surj}
The Hurewicz morphism is surjective and there are bijections 
$$
g \op H^* (EG \times _G X) (\field ^n) \cong \mathcal{Y}_n /G \cong [B
\field^n, EG \times _G X]_H.
$$
\end{cor}

\subsection{Hopkins--Kuhn--Ravenel generalized character theory}
\label{sect:gen-char}

This section exploits the work of Hopkins, Kuhn and Ravenel on
generalized group characters; for the convenience of the reader, a
rapid review of the theory is included. Refer to
the original source \cite{hkr} for the details. 

For $E$ a complex-oriented theory satisfying \fullref{hyp:E-hkr}, the ring $L (E^*)$ is the localization of the
continuous cohomology $\smash{E^* _{cont} (B \zed _p ^n)}$ inverting the
Euler classes corresponding to continuous morphisms $\zed ^n _p
\rightarrow S^1$ \cite{hkr}. There is a natural action of the group $\mathrm{Aut} (\zed_p ^n)$ of continuous automorphisms upon $L(E^*)$. 

 The following result contains the fundamental faithful-flatness property of the ring $L(E^*)$, which allows the simplification of arguments by using faithfully flat descent.

\begin{prop}{\rm \cite{hkr}}\qua
The ring $L(E^*)$ satisfies the following properties:
\begin{enumerate}
\item
 $L(E^*)$ is a faithfully flat $(`p^{-1} E^*`)$--module;
\item
the ring of invariants 
 $L(E^*)^{\mathrm{Aut}(\zed_p^n)}$ is isomorphic to $p^{-1} E^*$;
\item
there is an isomorphism of $L(E^*)$--modules
$$
L(E^*) \otimes _{E^* } E^* (B \field ^k)
\cong 
\amalg_{\hom (\field^n , \field  ^k) } L (E^*).
$$
\end{enumerate}
\end{prop}

Let $\mathrm{Fix}_{n} (G, X)$ denote the $G \times
\mathrm{Aut}(\zed_p^n)$--space $\amalg_{a \in \hom (\zed_p^n , G) } X
^{\mathrm{Im}(a)}$. The ring of generalized class functions is
$$
Cl_{n} (G, X; L(E^*)) \co = L (E^* ) \otimes _{E^*} E^* (\mathrm{Fix}_{n} (G, X))^G,
$$
which is an $L(E^*)$--algebra with an $\mathrm{Aut}(\zed_p^n)$--action. The generalized character map is the morphism of $E^*$--algebras 
$$
\chi^G _{n} \co 
E^* (EG\times _G X) 
\rightarrow Cl_{n} (G, X; L(E^*) ) ^{\mathrm{Aut}(\zed_p^n)},
$$
which is induced by the product of the  morphisms $
B (\zed/p^m) ^n \times X ^{\mathrm{Im} (a)} 
\rightarrow EG \times _G X
$ induced by  $a\co  (\zed/p^m)  ^n \rightarrow G$.

\begin{thm}
{\rm \cite[Theorem C]{hkr}}\qua
\label{thm:thmC} For
$E$ a complex-oriented theory satisfying \fullref{hyp:E-hkr}, the generalized character map $\chi^G _{n}$ induces  isomorphisms:
\begin{eqnarray*}
p^{-1} E^* (EG\times _G X) 
&\stackrel{\cong}{\longrightarrow}&
Cl_{n} (G, X; L(E^*) ) ^{\mathrm{Aut}(\zed_p^n)};
\\
L(E^* ) \otimes _{E^*}  E^* (EG\times _G X) 
&\stackrel{\cong}{\longrightarrow}&
Cl_{n} (G, X; L(E^*) ).
\end{eqnarray*}
\end{thm}

\subsection{Separation}\label{subs:Separation}
Throughout this section $E, G$ and $X$ satisfy the hypotheses of \fullref{thm:thmC}.

It is necessary to show that the map $[B \field ^n , EG \times _G X]_H
\rightarrow [B \field ^n , EG \times _G X]_E$ is  injective; this is
equivalent to showing that the map $\gamma$  in the diagram 
$$
\xymatrix@C=-10pt{
[B \field ^n , EG {\times _G} X]_H
\ar[r]
\ar[rd]^\gamma
&
\hom_{E^*} (E^* (EG {\times _G} X) , E^* (B\field ^n))
\ar@{^(->}[d]
\\
&
\hom_{L(E^*)} (L(E^*){\otimes _{E^*}} E^* (EG {\times _G} X) ,L(E^*){\otimes
_{E^*}}  E^* (B\field ^n))
}
$$
is injective, since the vertical morphism is an injection, by faithful flatness of
$L (E^*) $ over $E^*$.

The set $[B \field ^n , EG \times _G X]_H$ identifies with
$\mathcal{Y}_n /G$ by \fullref{cor:Hurw-surj}  and the ring $L(E^*)\otimes _{E^*} E^* (EG \times
_G X)$ identifies with $Cl_{n} (G, X; L(E^*) )$ by \fullref{thm:thmC}.

\begin{lem}
\label{lem:surj-Y}
There is a surjection 
$$
Cl_{n} (G, X; L(E^*) )
\twoheadrightarrow 
L(E^*) \otimes _{E^*} E^* (\mathcal{Y}_{n} ) ^G
$$
which is induced by restriction to elementary abelian subgroups and the passage to  path-connected components.
\end{lem}

\begin{proof}
The inclusion of the subspace $\amalg _{a \in \hom (\field ^n, G)} X ^{\mathrm{Im}(a)}$ of $\mathrm{Fix}_{n}(G, X)$ induces a surjection $E^* (\mathrm{Fix}_{n}(G, X))\twoheadrightarrow E^* ( \amalg _{a \in \hom (\field ^n, G)} X ^{\mathrm{Im}(a)})$, which is a morphism of $G$--modules. 

There is a morphism of spaces $\mathcal{Y} _{n} \rightarrow \amalg _{a
  \in \hom (\field ^n, G)} X ^{\mathrm{Im}(a)}$ induced by choice of
basepoint in each path-connected component. The morphism has a retract
provided by the natural projection $1 \rightarrow \pi_0$ and it
preserves the $G$--action up to homotopy; hence the morphism  gives
rise to a  surjection 
$$E^*  ( \amalg _{a \in \hom (\field ^n, G)} X
^{\mathrm{Im}(a)})\twoheadrightarrow E^* (\mathcal{Y}_{n})$$
 of $G$--modules. 

The composite induces the required surjection, either by constructing a section as a morphism of $G$--modules or by using the exactness of the $G$--fixed point functor, as in the proof of \cite[Theorem C]{hkr}.
\end{proof}

\begin{rem}
The $L(E^*)$--module $L(E^*) \otimes _{E^*} E^* (\mathcal{Y}_{n} )$ is a permutation $G$--mod\-ule, free on basis indexed by the elements of $\mathcal{Y} _{n}$. It follows that the $G$--invariants form a free $L(E^*)$--module  indexed over the orbits $\mathcal{Y}_{n} / G$. 
\end{rem}

\begin{lem}
The morphism $\gamma$ factorizes as 
$$
\xymatrix@C=-3pt{
[B \field ^n , EG {\times _G} X]_H
\ar[r]^(.27){\tilde \gamma}
\ar[rd]_\gamma
&
\hom_{L(E^*)} (L(E^*) {\otimes _{E^*}} E^* (\mathcal{Y}_n ) ^G  ,L(E^*){\otimes
_{E^*}} E^* (B\field ^n))
\ar@{^(->}[d]
\\
&\hom_{L(E^*)} (L(E^*){\otimes _{E^*}} E^* (EG \times _G X) ,L(E^*){\otimes
_{E^*}}  E^* (B\field ^n))
}
$$
in which the  inclusion is induced by the morphism of  \fullref{lem:surj-Y}.
\end{lem}

\begin{proof}
The result follows from the definition of the generalized character map. 
\end{proof}

Hence, the proof is completed by the following result.

\begin{prop}
\label{prop:gammabar-mono}
The morphism
$
\tilde{\gamma}
$
 is a monomorphism. 
\end{prop}

\begin{proof}
The $L(E^*)$--module $L(E^*) \otimes _{E^*} E^* (\mathcal{Y}_{n} )^G$ is a free $L(E^*)$--module on basis indexed by the orbits $\mathcal{Y}_{n} /G$. For $\bar{y}$ an orbit of $\mathcal{Y}_{n} /G$, the morphism $\tilde{\gamma} (\bar{y})$ is nontrivial and factors across the projection  onto the direct summand indexed by $\bar{y}$. The required injectivity follows.   
\end{proof}


\bibliographystyle{gtart}
\bibliography{link}
 
\end{document}